\documentclass[12pt]{amsart}

\usepackage{amsmath, amsthm, amscd, amssymb, graphicx}

\newtheorem{thm}{Theorem}[section]
\newtheorem*{thm*}{Theorem}
\newtheorem{cor}[thm]{Corollary}
\newtheorem*{cor*}{Corollary}
\newtheorem{lem}[thm]{Lemma}
\newtheorem{prop}[thm]{Proposition}

\newtheorem*{con*}{Conjecture}

\theoremstyle{definition}
\newtheorem{defn}[thm]{Definition}
\newtheorem*{ack}{Acknowledgements}
\newtheorem*{ob}{Observation}
\newtheorem{question}{Question}

\theoremstyle{remark}
\newtheorem{rem}[thm]{Remark}

\newtheorem*{example*}{Example}

\newcommand{\R}{\mathbb{R}}
\newcommand{\Z}{\mathbb{Z}}

\begin{document}

\title{Coboundary Expanders}

\author[Dominic Dotterrer]{Dominic Dotterrer}
\email{dominicd@math.toronto.edu}
\address{Department of Mathematics, University of Toronto} 

\author[Matthew Kahle]{Matthew Kahle}
\email{mkahle@math.ias.edu}
\address{School of Mathematics, Institute for Advanced Study}

\begin{abstract}
We describe a higher-dimensional generalization of edge expansion for graphs which applies to arbitrary cell complexes.  This generalization relies on a type of {\it co}-isoperimetric inequality.  We utilize these inequalities to analyze the topological and geometric behavior of some families of random simplicial complexes. \\

\end{abstract}

\maketitle
\section{Introduction}

The study of expander graphs is by now ubiquitous in mathematics and computer science. For a survey of this vast subject and numerous applications, see \cite{survey}.

Expanders graphs have arisen from several fields.  The first known examples of expanders, due to Pinsker, were probabilistic \cite{P73}. Later Margulis \cite{M88} exhibited an expander family using groups with Kazhdan's property {\bf T}.  Independently, Lubotzky, Philips, and Sarnak \cite{LPS88}, gave explicit examples using Deligne's proof of the Weil conjecture.  Although the first explicit constructions relied on deep number-theoretic or group-theoretic facts, Pinsker's earlier observation was that a sequence of  random $d$-regular graphs already gives an edge-expanding family.

It has long been felt that there should be higher-dimensional analogues of the theory of expander graphs.  Indeed this has been approached from a few points of view already.

Li defined and constructed Ramanujan complexes, higher-dimensional analogues of Ramanujan graphs, based on certain spectral properties \cite{Li}.  Following Lafforgue's proof of the Ramanujan conjectures, Lubotzky, Samuels, and Vishne gave additional examples  \cite{LSV05}, \cite{LSV05-2}.

In \cite{gromov-sing} Gromov considered an expansion property based on filling inequalities and showed that they are sufficient to imply a certain geometric overlap property.  Several families of complexes exhibiting the geometric overlap property were described by Fox, Gromov, Lafforgue, Naor, and Pach in \cite{FGLNP10}, including certain Ramanujan complexes studied earlier as well as new random models.  For an expository account of recent progress in this area see \cite{MW11} or the forthcoming \cite{dominic1}.

The aim of this article is twofold.  First we define a natural generalization of edge expansion from graphs to regular CW complexes.  This generalization is achieved by appealing to what we call ``co-isoperimetry'' since it is defined in terms of the coboundary operator.  The definition is similar to Gromov's filling inequalities --- we compare and contrast the definitions in the final section.

Second, we exhibit several examples of random simplicial complexes which have strong expansion properties with high probability.  In preparing this article it came to our attention that Gromov suggested in \cite{gromov-sing} (section 2.14) that random polyhedra should exhibit a weak form of co-isoperimetry.  However we obtain stronger inequalities here, and give several explicit examples .\\

The format of the rest of this article is as follows.\\

In section \ref{CE}, we review edge expansion of graphs and define coboundary expansion of cell complexes, which is a natural generalization of expansion to higher-dimensional cell complexes.  We also include some comments on geometric and topological interpretations of coboundary expansion.\\

Section \ref{RC} is devoted to introducing random simplicial complexes in particular the Erd\H{o}s-R\'enyi random graph, and its higher-dimensional analogue, the Linial-Meshulam random complex \cite{LM06} \cite{MW09}.  We show in later sections that Linial-Meshulam complexes provide examples of higher-dimensional expanders.  These complexes have well-studied topological and geometric behavior (\cite{BHK11}, \cite{K10}, \cite{ALLM10}), and they provide a template for studying a broader family of random complexes.\\

In section \ref{RS}, using a probabilistic construction, we show that certain kinds of random cell complexes provide expanding families, which is our main result.\\\

In section \ref{DP}, we exhibit several examples of random complexes satisfying the hypothesis of our main theorem. \\

In the final section, we make concluding remarks and pose questions for further study. \\

\begin{ack}
We acknowledge Larry Guth for several helpful conversations.  We also thank the Institute for Advanced Study in Princeton where this work was completed.\\
\end{ack}

\section{Coboundary expansion}\label{CE}

In this section, we discuss a natural generalization of edge expansion to higher-dimensional cell complexes.  We begin by recalling the definition of the edge expansion.

\begin{defn}
For a finite graph $G$ the {\it edge expansion}, $h(G)$, is defined by
$$h(G) := \min_{A \subset G}   \;  \frac{\#E(A,B)}{ {\rm min} \{ |A|, |B| \} } $$
where $\#E(A,B)$ is the number of edges which connects $A$ to its complement $B= G \setminus A $.
\end{defn}

Here $A$ is understood to be a proper nonempty subset of $G$.  The edge expansion, $h(G)$ is also commonly referred to as the {\it Cheeger constant} in analogy with Riemannian geometry.  From the definition, a finite graph is connected if and only $h(G) >0$.

One is often more interested in edge expansion of an infinite family of graphs rather than for any particular graph.

\begin{defn}
A family of graphs $\{ G_n \}_{n \ge 1} $ of bounded degree is an {\it expander family} if $|G_n| \to \infty$ and $\liminf h(G_n) >0$ as $n \to \infty$.
\end{defn}

Although it is common practice to study edge expansion for a sequence of $d$-regular graphs with $d$ fixed, we need not be restricted to regular graphs or even graphs of bounded degree.  Instead, we can compare the edge expansion of $G$ to the {\it maximum degree} $D (G)$ to give a more general definition.

\begin{defn}
A family of graphs $\{ G_n \}^\infty_{n=1} $, is a {\it degree-relative expander family} if $|G_n| \to \infty$ and $\liminf \frac{h(G_n)}{ D(G_n)} >0$ as $n \to \infty$.\\
\end{defn}

\begin{rem} From the definitions it is clear that $\frac{h(G_n)}{ D(G_n)} \le 1$, which suggests that degree might be a reasonable thing to compare with the Cheeger constant.
\end{rem}

\begin{rem}
Examples of degree-relative expander families are the complete graph $K_n$ and the Erd\H{o}s-R\'enyi random graph $G(n,p)$ when $ p \gg \log{n} / n$, as discussed below.
\end{rem}

We give equivalent formulations of these definitions below.  Before we do so, we introduce some notation.

Let $X$ be a polyhedral complex.  We denote the set of $k$-dimensional cells of $X$ as $X^{(k)}$, and the $\Z_2$-vector space of $k$-dimensional cochains by,
$$C^k X := \{ X^{(k)} \to \Z_2 \} $$ 

It is equipped with a natural exterior differential calculus, that is $\Z_2$-linear maps,
$$d : C^k X \to C^{k+1} X \qquad \text{   where   }  \quad d\beta (x) = \sum_{y \in \partial x} \beta (y)$$
and $\partial x$ denotes the boundary of the $(k+1)$-cell $x$.

As usual let $Z^k X$ denote the subspace of cocycles and $B^k X$ the subspace of coboundaries.  Cohomology is defined by $H^k(X) = Z^k X / B^k X$.

\begin{defn}We equip $C^k X$ with the norm, $$||\beta|| := {\rm supp } \beta,$$ and we also refer to the quotient norm with respect to the coboundary map $d$ which we denote $||[\cdot]||$, i.e. $$||[\beta]|| :=  \inf_{\alpha \in C^{k-1} X} || \beta + d \alpha ||.$$
\end{defn}

A more standard notation for quotient norm might be $|| \beta ||_{C^k / B^k}$ but we use $|| [ \beta ] ||$ here for the sake of simplicity.

\begin{rem}  We use $\Z_2$ coefficients throughout, and measure the size of a cochain by its support.  This is natural for our applications, and gives an analogue of edge expansion.  However the role of other coefficients (and other choices of norm) is quite interesting and we discuss it in more detail in the last section.
\end{rem}

\begin{ob}
Consider a set, $A$, of vertices in a graph $G$ as a $\Z_2$-cochain $\beta \in C^0 G$ (i.e. $A = {\rm supp} \beta \subset G$).  Then the differential, $d \beta$ is supported on the edges which connect $A$ to its complement.  Working in the reduced co-chain complex, we see that $B^0 G $ has exactly two elements: one supported on $G$ and one supported on $\emptyset$. With these considerations we can rewrite the edge expansion as
$$h(G) = \min_{\beta \in C^0 G \setminus B^0 G} \; \frac{||d \beta||}{|| [ \beta ] || }.$$
\end{ob}

Now that we have reformulated the definition of edge expansion purely in terms of the normed exterior differential calculus, we are in a position to generalize it in a direct way.

\begin{defn}
Let $X$ be a regular CW complex.  The $k$th {\it coboundary expansion} of $X$ is defined by
$$h^k (X) := \min_{\beta \in C^{k} X \setminus B^k X} \frac{ ||d \beta|| }{|| [ \beta ] ||}$$
\end{defn}

An important property of coboundary expansion is that $h^k (X) = 0 $ if and only if $\widetilde{H}^{k} (X) \neq 0 $.  To see this note that $|| d\beta ||= 0$ if and only if $\beta$ is a cocycle, and $|| [\beta] ||= 0$ if and only if $\beta$ is a coboundary.  So we only have $h^k (X)=0$ if there is some $k$-cocycle which is not a coboundary, i.e.\ a nontrivial element of $k$th cohomology.

Note that this definition of coboundary expansion has the following consequence.

\begin{ob}
Suppose that a cell complex $X$ has $k$th coboundary expansion $h^k (X) $.  If we were to try to create a nontrivial cohomology class $[ \beta ] \in H^{k} (X)$ by deleting some $(k+1)$-dimensional cells, and we would like that nontrivial class to be large: $|| [\beta] || \geq m$,  then we would need to delete from $X$ at least $h^k(X) \cdot m$ cells of dimension $k+1$.  

\end{ob}

\begin{defn}
Let $X$ be a cell complex.  The maximum $k$-degree $D_k(X)$ is the maximum number of $(k+1)$-dimensional faces containing a $k$-face.
\end{defn}

We are particularly interested in infinite families.

\begin{defn}

A family of polyhedral complexes $X_n$ is called a {\it degree-relative} $k$-{\it coboundary expander family} if $|X^{(k)}_n| \to\infty$ and $\frac{h^k (X_n) }{D_k (X) } \geq c > 0$ for all $n \geq 1$. 

\end{defn}

One of the main points of this article is to exhibit families of fairly sparse degree-relative $k$-coboundary expanders.  (By ``fairly sparse'' we mean not bounded of degree, but of degree growing only logarithmically in the number of vertices.) We construct these families probabilistically.

\section{Random complexes}\label{RC}

In this section, we  introduce our main object of study: random polyhedral complexes.  In order to properly motivate them, we  begin with one of the most well-studied random graph models due to Erd\H{o}s and R\'enyi. 

\begin{defn} The {\it Erd\H{o}s-R\'enyi random graph} 
$G(n,p)$ is the probability space of all graphs on vertex set $[n] = \{ 1, 2, \dots, n \}$ with each edge having probability $p$, independently.  In other words for every graph $G$ on vertex set $[n]$, $$\mathbb{P}[G \in G(n,p)] = p^{e(G)} (1-p)^{{n \choose 2} - e(G)},$$ where $e(G)$ denote the number of edges of $G$.
\end{defn}

We say that a sequence of probability spaces $\{ \Omega_n \}$ has property $\mathcal{P}$ {\it asymptotically almost surely} (a.a.s.) if the probability of $\mathcal{P}$ tends to $1$ as $n$ tends to infinity.\\

The following theorem is due to Erd\H{o}s and R\'enyi \cite{ER}.

\begin{thm} \label{thm:ERthresh} Let $\omega=\omega(n)$ be any function that tends to infinity with $n$.   If  $p = (\log{n}+ \omega) / n $ then $G(n,p)$ is a.a.s.\ connected, and if $p = ( \log{n} - \omega)/n$ then $G(n,p)$ is a.a.s.\ disconnected.
\end{thm}   

Once $p$ is much larger than $\log{n} / n$, $G(n,p)$ is connected \cite{ER} and it exhibits edge expansion.  The following theorem is due to Benjamini, Haber, Krivelevich, and Lubetzky \cite{ERexpand}.  Let $D(G)$ denote the maximum degree of $G$.

\begin{thm} \label{thm:ERiso} Let $0 < \epsilon < 1/2 $ be fixed.  Then there exists a constant $C=C(\epsilon)$ such that if $p \ge  C \log{n} / n $, then a.a.s.\ $G \in G(n,p)$ has Cheeger constant bounded by $$(1/2 - \epsilon) D(G)  \le h(G) \le (1/2+\epsilon) D(G).$$
\end{thm}

\medskip

\begin{rem} The results in \cite{ERexpand} are more precise than what we state here, since in particular the entire random graph process is considered. Moreover, how $C$ depends on $\epsilon$ can be made much more explicit. A small point is that the statements in \cite{ERexpand} are for minimum degree, but this is equivalent since in this regime all the vertices have roughly the same degree ($\approx pn$), so the same statement holds for maximum degree. \\
\end{rem}

\begin{rem} The upper bound on $h(G)$ in Theorem \ref{thm:ERiso} is straightforward.  If one splits the vertices in half randomly, one should expect about $p n^2 /4$ edges between the two halves.  Dividing by $n/2$ gives an upper bound of $pn/2$ , compared to the maximum degree which is only slightly larger than $pn$ \cite{Bollo}.  The lower bound  follow as a corollary of our main result.\\
\end{rem}

Linial and Meshulam defined $2$-dimensional analogues of $G(n,p)$ and proved a cohomological analogue of Theorem \ref{thm:ERthresh} \cite{LM06}, and Meshulam and Wallach extended the result to arbitrary dimension \cite{MW09}. 

Let $\Delta_n$ denote the $(n-1)$-dimensional simplex and $\Delta^{(i)}_n$ its $i$-skeleton.

\begin{defn}{The {\it Linial-Meshulam complex}}

The random simplicial complex $Y_k (n,p)$ is the probability space of all simplicial complexes with complete $k$-skeleton and each $(k+1)$-dimensional face appearing independently with probability $p$.  In other words for  every $$\Delta^{(k)}_n \subseteq Y \subseteq \Delta^{(k+1)}_n,$$ we have $$\mathbb{P}(Y \in Y_k (n,p) ) = p^{\lvert Y^{(k+1)} \rvert} (1-p)^{ \binom{n}{k+2} - \lvert Y^{(k+1)} \rvert },$$
\end{defn}

In particular, $Y_0 (n,p) \cong G(n,p)$.

\bigskip

The main result of \cite{LM06} (for $k=1$) and \cite{MW09} (for $k \ge 2$) is the following.

\begin{thm}   \label{thm:LMW} Let $\omega=\omega(n)$ be any function that tends to infinity with $n$.   If  $p = \frac{(k+1) \log{n}+ \omega }{ n} $ and $Y \in Y(n,p)$ then a.a.s.\ $H^k(Y,\Z_2)=0$ and if $p =\frac{(k+1)  \log{n} - \omega}{n}$ then a.a.s.\ $H^k(Y,\Z_2) \neq 0$.
\end{thm}

In the next two sections, we refine the statement that $H^k (Y)=0$ to a more quantitative geometric statement by estimating $h^k (Y)$.  At the same time we expand the statement to more general random complexes.

We now define {\it random} $p$-{\it subcomplexes} (which include the Linial-Meshulam complex as an example).

\begin{defn}
Let $\{ X_n \}$ be a family of finite regular CW complexes.  Then let $X^{(k)}_n \subset Y_{n,p} \subset X^{(k+1)}_n$ be a random subcomplex of $X_n$ chosen so that each $(k+1)$-cell of $X_n$ is included in $Y_{n,p}$ independently with probability $p$.  We refer to $Y = Y_{n,p}$ as a {\it random} $p$-{\it subcomplex of }$X_n$.
\end{defn}

Our main result is that if $\{ X_n \}$ has strong enough expansion properties, and if $p$ is large enough, then random $p$-subcomplexes of $\{ X_n \}$ also have strong expansion properties.  We make this statement precise and prove it in the next section.

\section{Main result}\label{RS}

In this section we show that random $p$-subcomplexes of polyhedra can inherit coboundary expansion from the ambient complex.  In particular, this provides a threshold beyond which random subcomplexes have vanishing cohomology with high probability.


\begin{defn}
Let $\{ X_n \}$ be a sequence of finite $(k+1)$-dimensional polyhedral complexes. We say that family $\{X_n\}$ is {\it face-relative expanding} if 
$$ \frac{\log |X^{(k)}_n |}{h^k(X_n)} \longrightarrow 0 \quad \text{    as     } n \to \infty$$
\end{defn}

In the next section, we give several examples of face-relative expanding families of polyhedra, which include in particular the simplex of dimension $n$.


We now proceed to our main result.

\begin{thm}\label{main}
Let $\epsilon > 0$ and let $\omega=\omega(n)$ be any function such that $\lim_{n \to \infty} \omega = \infty$.  Let $Y_{n,p}$ be a random $p$-subcomplex of a face-relative expanding family of polyhedra $X_n$.  If $$p \geq \frac{ 2 \log |X^{(k)}_n| + \omega }{\epsilon^2 h^k(X_n)}$$ then $$h^k (Y_{n,p}) \geq (1-\epsilon) p \cdot h^k(X_n)  \quad \text{ a.a.s. }$$

\end{thm}

\begin{proof}

Let $\beta \in C^{k} X_n$ be a $k$-cochain in $X_n$ (and hence also in $Y_{n,p}$ since $Y$ contains the entire $k$-skeleton of $X_n$).  We let $\Vert d \beta \rVert_Y$ denote the norm of $d \beta \in C^{k+1} Y$, as distinguished from simply $\lVert d \beta \rVert$ which denotes in norm in the full complex $X_n$.  Since each $(k+1)$-cell is included in $Y$ independently with probability $p$, by the Chernoff-Hoeffding bounds \cite{chernoff52} (or see, for example, \cite{concentration09}),

$$\mathbb{P} \bigl[ \lVert d\beta \rVert_Y \leq (1-\epsilon) p \lVert d\beta \rVert \bigr] \leq e^{-\frac{\epsilon^2}{2} p\cdot \lVert d\beta \rVert} $$\\
Let $n = \lVert d \beta \rVert$ and $m = \lVert [\beta] \rVert = {\rm min}_\alpha \lVert \beta + d \alpha \rVert$, so that $n \geq h^k(X_n) \cdot m$.  We must now check that the inequality $$\lVert d\beta \rVert_Y \geq (1-\epsilon) p \cdot h^k(X_n) \Vert[ \beta ]\rVert$$ \\
holds for every $\beta \in C^{k} Y$.  We note that it suffices to check this inequality on all $\beta$ for which $\lVert \beta \rVert = \lVert [\beta] \rVert$.  Indeed, if the inequality holds on all such {\it minimizing} $\beta$, it then holds for {\it all} $\beta$. \\

Here we apply a union bound.  We can upper bound the number minimizing $\beta$ of  $m$ by counting {\it all} $k$-cochains of norm $m$.  

$$\mathbb{P} \bigl[ \exists \beta, \; ||d \beta ||_Y \leq (1-\epsilon) p n \bigr] \leq \sum_{m \geq 1} \binom{ |X^{(k)}_n|}{ m} e^{-\frac{\epsilon^2 p m \cdot h^k(X_n)}{2} } $$

$$= \biggl[ 1 + e^{-\frac{\epsilon^2 p \cdot h^k(X_n)}{2} } \biggr]^{|X^{(k)}_n|} - 1$$

In order to show that this quantity goes to $0$, we need to show that the leftmost quantity goes to $1$, or simply that $$|X^{(k)}_n| \log \bigl[ 1 + e^{-\frac{\epsilon^2 p \cdot h^k(X_n)}{2} } \bigr] \longrightarrow 0$$

Now, $$\log \bigl[ 1 + e^{-\frac{\epsilon^2 p \cdot h^k(X_n)}{2} } \bigr] \leq e^{-\frac{\epsilon^2 p \cdot h^k(X_n)}{2} },$$ and finally, since $$p \geq \frac{ 2 \log |X^{(k)}| + \omega }{\epsilon^2 h^k(X_n)}, $$ we have, $$  |X^{(k)}_n|  e^{-\frac{\epsilon^2 p \cdot h^k(X_n)}{2} }  \leq  e^{-\frac{\omega}{2}} \longrightarrow 0.$$

\end{proof}

\begin{rem}The theorem above provides, in particular, a threshold for $p$ beyond which a random $p$-subcomplex a.a.s.\ has vanishing cohomology, $\tilde{H}^{k} (Y_{n,p} ) = 0$.  For all of the examples in the next section, this is sharp up to a constant factor.
\end{rem}

\section{Examples}\label{DP}

In this section we give three examples of face-relative expanding families of simplicial complexes.  Their random $p$-subcomplexes give degree-relative expanders.

\subsection{The Simplex}

\begin{prop}Let $\Delta_n$ be the simplex on $n$ ordered vertices.  Then $$h^k (\Delta_n) = \frac{n}{k+2}.$$
\end{prop}
To the authors' best knowledge, the proof of this proposition first appeared in \cite{MW09} and later independently in \cite{gromov-sing}.  The proof method is analogous to the standard technique for proving the linear isoperimetric inequality in the round sphere --- this analogy is discussed in the expository article \cite{dominic1}.

By Theorem \ref{main} we have the following.

\begin{cor} 
Fix $0 < \epsilon < 1$ and suppose $Y_k (n,p)$ is a random $p$-subcomplex in $\Delta_n$, with $$p \geq \frac{2(k+1)(k+2) \log n + \omega}{\epsilon^2 n}.$$
Then $Y$ a.a.s. satisfies $$h^k (Y) \geq (1-\epsilon) p\frac{n}{k+2}.$$ 
On the other hand the maximum degree satisfies $D_k (Y) \approx np$, so $Y$ is a.a.s.\ a degree-relative expander.\\
\end{cor}

\begin{rem}  This shows in particular that $H^k(Y) = 0$ for $$p \ge \frac{2(k+1)(k+2) \log{n}+ \omega}{n} .$$  This recovers one direction of Theorem \ref{thm:LMW}, up to a constant factor.
\end{rem}

\begin{rem} The lower bound on $h^k$ is essentially tight.  In particular it is easy to check that a.a.s.\ $h^k(Y) \le (1 + \epsilon) p \frac{n}{k+2}$.\\
\end{rem}

\subsection{The Cross-polytope}

\begin{prop} \label{cross} Let $\widehat{Q}_n$ denote the $n$-dimensional cross-polytope. Then $$h^k (\widehat{Q}_n) \geq \frac{2(n-k-1)}{k+2}.$$\\
\end{prop}

Let $Q_n$ denote the $n$-dimensional cube.  By duality, we have the canonical isomorphism $\iota: C^k  \widehat{Q}_n \to C_{n-k-1} Q_n$ and the following diagram commutes.

$$\begin{CD}
\cdots @>>> C^{k} \widehat{Q}_n @>{d}>> C^{k+1} \widehat{Q}_n @>>> \cdots \\
@. @V{\iota}VV @VV{\iota}V @.\\
\cdots @>>> C_{n-k-1} Q_n @>{\partial}>> C_{n-k-2} Q_n @>>> \cdots\\
\end{CD}$$\\

Therefore, each cochain $\beta \in C^{k} \widehat{Q}_n$ can be identified as an $(n-k-1)$-chain in $Q_n$ and we seek a minimal $(n-k-1)$-chain which fills the boundary $\partial \beta$.  We first prove an isoperimetric inequality in $Q_n$.\\

\begin{lem}\label{cube}
Let $z $ be a $j$-cycle in the cube (with $\Z_2$ coefficients), then  there exists a $(j+1)$-chain $y$, with $\partial y = z$ and $${\rm vol}_{j+1} (y)  \leq \frac{n-j}{2(j+1)} {\rm vol}_j (z) $$
\end{lem}

\begin{proof}[Proof of Lemma \ref{cube}]

Let $C_{M,j}$ be the optimal isoperimetric constant for $j$-cycles in $Q_M$.  

We inductively bound $C_{M,j}$
\bigskip

If $M= j+1$, then the set of cycles in $C_j Q_{j+1}$ consists of the $1$-dimensional subspace spanned by the boundary of $Q_{j+1}$.  Therefore, $C_{j+1,j} = \frac{1}{2(j+1)}$
\bigskip

Now, allow us to suppose that $C_{M,k}$ has been bounded from above for all $j+1 \leq M < n$.
\bigskip

Let $H_+$ be an $(n-1)$-dimensional face of $Q_n$, and let $H_-$ denote the opposite face (or ``shadow," as it is sometimes called).  Each $j$-face, $x \in Q^{(j)}_n$, is contained in $n-j$ faces of dimension $n-1$.  Hence $$\sum_{H_+ \in Q^{(n-1)}_n} {\rm vol}_j (z \cap H_+)= (n-j) {\rm vol}_j (z) $$

Now, $(z \cap H_{+}) - (z \cap H_{-})$ can be identified with a $j$-cycle in $H_{-} \cong Q_{n-1}$.  Its volume is no greater than ${\rm vol}_j (z \cap H_{+} ) + {\rm vol}_j (z \cap H_{-})$.  We fill this cycle using the induction hypothesis, i.e. there exists $y_- \in C_{j+1} H_-$ such that $\partial y_- = (z \cap H_{+}) - (z \cap H_{-})$ and 
 $${\rm vol}_{j+1} (y_-) \leq C_{n-1, j} [{\rm vol}_j (z \cap H_{+} ) + {\rm vol}_j (z \cap H_{-})]$$

We also consider the chain $z \cap H_+ \in C_j H_+$.  Taking the product of this chain with the unit interval, we have a natural identification of $(z \cap H_+ ) \times {\bf I} \in C_{j+1} Q_N$.  We will denote this $j+1$-chain as $y_+$.
\bigskip

We will take $y$ to be $y = y_+ - y_-$.  Now, we simply check that $$\partial y = (z \cap H_+) + (\partial (z\cap H_+) \times {\bf I} ) + ( z \cap H_-) = z$$

Now we have constructed a $(j+1)$-chain, $y$, satisfying $\partial y = z$ and

$$\begin{array}{lll}

{\rm vol}_{j+1} (y) &\leq& {\rm vol}_j (z \cap H_+) + C_{n-1,j} {\rm vol}_j [(z \cap H_+ ) - (z \cap H_-)] \\

&\text{  }&\\

&\leq& {\rm vol}_j (z \cap H_+) + C_{n-1,j} [{\rm vol}_j (z \cap H_+ ) + {\rm vol_j} (z \cap H_-)]\\

\end{array}$$

Summing over choices of $H_{+}$ yields,

$$2n {\rm vol} (y) \leq (n-j) {\rm vol}_j (z)  + 2 (n-j)  C_{n-1, j} {\rm vol}_j (z) $$

$$ \Rightarrow C_{n,j} \leq \frac{n-j}{2n} (1 + 2C_{n-1,j})$$

Now one can check that $C_{n,j} = \frac{n-j}{2(j+1)}$ is a solution to the recursive formula above with $C_{j+1,j} = \frac{1}{2(j+1)}$

\end{proof}

\begin{proof}[Proof of Proposition \ref{cross}]
Suppose  $\beta$ is a $k$-cochain corresponding to an $(n-k-1)$-chain in the cube.  Therefore, let $j= n-k-2$ and we see that there exists $\alpha \in C^{k-1} \widehat{Q}_n$ such that
$$||\beta + d\alpha || \leq \frac{k+2}{2(n-k-1)} ||d \beta ||$$
\end{proof}

Now we have the following.

\begin{cor}
Let $ V_{n,p}$ be a random $p$-subcomplex of the $n$-dimensional cross-polytope where $(k+1)$-dimensional faces are chosen independently with probability $p$.  Suppose that $$p \geq \frac{(k+2)(k+1) \log n + \omega}{\epsilon^2(n-k-1)},$$
where $\omega \to \infty$ as $n \to \infty$. Then $V \in {V}_{n,p}$  a.a.s.\ satisfies 
$$h^k (V) \geq  \frac{2(1-\epsilon)(n-k-1)p}{k+2}   .$$
On the other hand the maximum $k$-degree $D_k$ satisfies $D_k \approx np$ so $V$ is a.a.s.\ a degree-relative expander. 
\end{cor}

\begin{proof}

We simply note that $$|\widehat{Q}^{(k)}_n |= |Q^{(n-k-1)}_n| = \binom{n}{k+1} 2^{k+1} \quad $$
$$\Longrightarrow \quad \log |\hat{Q}^{(k)}_n | \leq (k+1) \log n + (k+1) \log 2,$$

\bigskip
and the proposition follows.
\end{proof}

\subsection{The Complete $(k+2)$-partite Complex}

Let $\Lambda_{n,k}$ be the complete $(k+2)$-partite $k$-complex with $(n, \dots, n)$ vertices.  In other words, $$\Lambda_{n,k} := \overset{k+2}{\overbrace{ [n] * \cdots * [n]}}$$ where $*$ denotes join and $[n]$ denotes a set of of $n$ vertices.  

\begin{prop}
We have $$h^{k} (\Lambda_{n,k}) \geq \frac{n}{2^{k+1} - 1}.$$
\end{prop}
\bigskip

 For example, when $k=0$, $\Lambda_{n,0}$ is the complete bipartite graph, $K_{n,n}$.  The Cheeger constant of the graph $K_{n,n}$ is given by $h(K_{n,n}) = n$.

\bigskip

\begin{proof}

The proof is by induction on $k$.

\bigskip

We will suppose that for all $\beta \in C^{k-1} \Lambda_{n,k-1}$, $$\frac{ \Vert d \beta \rVert}{\lVert [\beta]\rVert } \geq c_{k-1}.$$

\bigskip

There are two types of $k$-faces in $\Lambda = \Lambda_{n,k}$.  We say that a $k$-face, $[v_0, \dots, v_k]$, is in type one, $\mathcal{T}_1$, if it contains a vertex from the first factor $[n]$ of the simplicial join; there are $k n^{k+1} $ many faces of type $1$.  We say that $[v_1, \dots, v_k]$ is of type $2$, $\mathcal{T}_2$, if all its vertices come from the second through $(k+1)$-th factors of the join; there are $n^{k+1}$ many faces of type $2$.

\bigskip

For example, in a $(k+1)$-dimensional face $[v_0 , \dots, v_{k+1} ]$, the subface $[\hat{v_0} , v_1, \dots v_{k+1} ]$ is of type $2$, whereas the subfaces  $ [ v_0, \dots, \hat{v_i} , \dots, v_{k+1} ]$ are all of type $1$.

\bigskip

Let $\beta \in C^k \Lambda$.  Let $v,w \in [n]$ be two vertices in the first factor of the join.  Define $\eta_{w,v} \in C^{k-1} \Lambda_{n, k-1}$ by $$\eta_{w,v} [ v_1, \dots, v_{k} ] := \beta [ w, v_1, \dots , v_{k} ] - \beta [ v, v_1, \dots, v_{k}] .$$

For each $\eta_{w,v}$, we apply the induction hypothesis to find $\alpha_{w,v} \in C^{k-2} Z_{n, k-1}$ such that $$ c_{k-1}||\eta_{w,v} + d\alpha_{w,v} || \leq ||d\eta_{w,v}||.$$  Using this we will define $\zeta_v \in C^{k} \Lambda_{n,k}$ as

$$\zeta_v [v_0, \dots, v_k] := \left\{
\begin{array}{cl}
  (\eta_{v_0, v} + \alpha_{v_0, v}) [ \hat{v_0}, \dots, v_k] , & \text{if}\; [v_0, \dots , v_k] \in \mathcal{T}_1\\
d \beta [ v, v_0, \dots, v_k ] , & \text{if}\; [v_0, \dots , v_k] \in \mathcal{T}_2
\end{array}\right.$$

It is straightforward to check that $d\zeta_v = d\beta$.

\bigskip

Now we average the norms of the $\zeta_v$:

$$\frac{1}{n} \sum_v \lVert \zeta_v \rVert = \frac{1}{n} \Bigl(    \lVert d \beta \rVert + \sum_{v,w} \lVert \eta_{v,w} + d\alpha_{v,w} \rVert \Bigr) $$

$$\leq \frac{\lVert d\beta \rVert}{n} + \frac{1}{n \cdot c_{k-1}} \sum\limits_{v \neq w}    \Bigl\lVert d\beta [ w, \cdot ] - d\beta [ v, \cdot ] \Bigr\rVert$$

$$\leq \frac{\lVert d\beta \rVert }{n}  + \frac{1}{n \cdot c_{k-1}} \sum_{v}  \Bigl( \lVert d \beta \rVert - \lVert d\beta[v,\cdot] \rVert \Bigr)$$

$$+ \frac{n-1}{n \cdot c_{k-1}} \sum_v \lVert d\beta [v, \cdot ] \rVert $$

$$= \frac{1}{n} \Bigl( 1 + \frac{2n - 2}{c_{k-1}} \Bigr) ||d \beta||$$

Therefore we can set $$\frac{1}{c_k} := \frac{1}{n} \biggl( 1 + \frac{2n - 2}{c_{k-1}} \biggr) $$

Finally, noting that $c_0 = n$, we can check that $c_k \geq \frac{n}{2^{k+1} - 1}$:

$$c_k = n \biggl( \frac{ c_{k-1}   }{c_{k-1} + 2n -2  } \biggr) \geq  \frac{ n^2}{ n + (2^{k} - 2) n}  = \frac{n}{2^k - 1}$$

That is, $$h^k (\Lambda) \geq  \frac{n}{2^{k+1} - 1}$$

\end{proof}

\begin{cor}

Fix $k$, and let $W_{n,p}$ be a random $p$-subcomplex of the complete $(k+2)$-partite $k$-complex on $(n,\dots, n)$ vertices with $$p \geq \frac{ 2(2^{k+1} - 1) (k+1) \log n + \omega}{\epsilon^2 n}.$$
Then $W \in W_{n,p}$ a.a.s.\ satisfies $$h^{k} (W) \geq \frac{(1-\epsilon) np}{2^{k+1} -1}.$$
We also have $D_k (W) \approx np$ so $W$ is a.a.s.\ a degree-relative expander.
\end{cor}

\section{Concluding Remarks}
\begin{enumerate}

\item In \cite{gromov-sing}, Gromov proposes a definition of expansion very similar to the one given in section \ref{CE}.

\begin{defn}\cite{gromov-sing}
For a polyhedral complex $X$, the $k${\it -th  filling norm}, $\lVert d^{-1}_k \rVert_X$ of $X$ is defined as the smallest constant $c$ so that for every $(k+1)$-coboundary $d \beta \in B^{k+1} X$, there exists $\alpha \in C^k X$ such that $d \alpha = d \beta$ and $$\frac{\lVert \alpha \rVert }{ \lvert X^{(k)} \rvert} \leq c \frac{\lVert d\beta \rVert}{\lvert X^{(k+1)} \rvert}.$$
\end{defn}

It is easy to check that when $H^k (X) = 0$, $$h^k (X) = \frac{\lvert X^{(k+1)} \rvert}{\lVert d^{-1}_k \rVert_X \cdot  \lvert X^{(k)} \rvert}.$$
However, this definition has a significant difference from our coboundary expansion in that $$H^k (X) = 0 \Leftrightarrow h^k (X) > 0,$$ while, on the other hand it is possible for both $\lVert d^{-1}_k \rVert_X < \infty$ and  $H^k (X) \neq 0$.\\
So each definition probably has its advantages.  The definition of $k$th coboundary expansion presented here can detect whether $H^k=0$, and seems to be the most straightforward generalization of edge expansion.  On the other hand, Gromov's definition of filling norm may still give some geometric information in the case that $H^k \neq 0$.\\

\item In each of the above examples we have $h^k (X_n) = \Omega (n) $ while $\log |X^{(k-1)}_n| = O (\log n)$.  In particular, all of these complexes exceed the condition that $h^k (X) = \omega ( \log |X^{(k-1)}| )$.

On the other hand, the $n$-dimensional hypercube, $Q_n$ does not satisfy this condition.  Indeed, $h^k (Q_n) \equiv 1$ (see, for example, \cite{gromov-sing}), while at the same time $$\log |Q^{(k-1)}_n | = \log \Bigl[ \binom{n}{k-1} 2^{n-k+1} \Bigr]= \Omega ( n \log n ).$$
So in particular the cube is not a face-relative expander family.\\

\item  We have discussed cohomology with $\Z_2$ coefficients, which are combinatorially convenient and which provide a natural generalization of edge expansion.  The computations in Section \ref{DP} of the $k$th coboundary expansion of the the simplex, cross-polytope, etc., can be adapted for arbitrary coefficients with $L^1$ norm.  However the proof of the main result in Section \ref{RS} does not extend to $\R$ or $\Z$ coefficients, which are of particular interest.

In the case of graphs, $\R$-expansion with Euclidean norm)is related to the spectral gap of the graph Laplacian as follows.
$$\sqrt{\lambda_1} = h^0 (G ; \R) = \min_{\beta \in C^0 (G; \R) } \frac{ \lVert d\beta \rVert_2 }{\min_c \lVert \beta + c \cdot 1_G \rVert_2 } $$ where $1_G$ is the function which is identically $1$ on the vertices of $G$, and  $\lVert  \cdot \rVert_2$ denotes the standard Euclidean norm.

The spectral gap is in turn related to $\Z_2$ expansion by the Cheeger and Buser inequalities.

\begin{thm}\label{cheeger} \cite{D84} \cite{AM85} \cite{A86}

Let $G$ be a graph with maximum degree $D= D (G)$. Then
$$\frac{\lambda_1}{2} \leq h^0(G; \Z_2) \leq \sqrt{2D\lambda_1}.$$
\end{thm}

%
%
%
%
%
%

Putting together these facts, we can relate $\Z_2$-expansion and $\R$-expansion for graphs, but it is still unclear whether higher-dimensional analogues of this hold.

\begin{question}  Are there higher-dimensional analogues of the Cheeger and Buser inequalities?  In particular, can one bound the spectral gap of the higher order Laplacians in terms of  $\Z_2$-expansion?
\end{question}

It is worth noting that \cite{gromov-sing} Gromov obtained a relation between $h^1 (X ; \R)$ and  $h^1 (X ; \Z)$, where the first is defined by the Euclidean norm while the second by the $L^1$ norm.\\

\item Most often the term {\it edge expander family} refers to a family of graphs $\{ G_n \}$ of bounded degree.  Random $d$-regular graphs are perhaps the most canonical example; see \cite{Wormald} for a general reference.  However we do not know any such examples in higher dimension.

\begin{question}
Do there exist families of polyhedral complexes $\{ X_n \}$ with $D_k (X_n) $ bounded and $\liminf h^k (X_n ; \Z_2)  > 0$?
\end{question}
Uli Wagner recently noted that Linial-Meshulam random complexes with $p=C / n$ and sufficiently large $C$ are in a certain sense {\it coarse expanders} \cite{Uli}.  These are bounded average degree but not bounded degree. The Ramanujan complexes studied in \cite{Li} \cite{LSV05} \cite{LSV05-2} are of bounded degree, and satisfy certain spectral gap conditions, but it is not clear to us if they satisfy $\Z_2$-expansion in the sense discussed here.
\end{enumerate}

\bibliography{expand_refs.bib}
\bibliographystyle{plain}

\end{document}